%% file: Many_cliques_in_bounded-degree_hypergraphs_arXiv2.tex
\documentclass[12pt]{article}

\usepackage[paper=letterpaper,margin=0.7in,twoside=false,includehead]{geometry}

\usepackage{amsfonts,amsmath,amsthm, amssymb} 
 
\usepackage{hyperref}
\usepackage[capitalize]{cleveref}

\usepackage{graphicx}
\usepackage{fancyhdr}
\usepackage{xparse}
\usepackage{mathtools}

\usepackage[shortlabels]{enumitem}
\usepackage[normalem]{ulem}

\SetEnumitemKey{thmparts}{topsep=2pt plus 1pt minus1pt, itemsep=1pt plus0.5pt minus0.5pt}

\usepackage{xcolor}
\usepackage{verbatim}

\usepackage{thmtools}
\usepackage{thm-restate}

\usepackage{etoolbox}

\usepackage[title]{appendix}

\input basics

\DeclareDocumentCommand{\shad}{O{q}}{\partial_{#1}}
\DeclareDocumentCommand{\upshad}{O{t}}{U^{#1}}

\DeclareDocumentCommand{\k}{O{t}}{k^{#1}}
\DeclareDocumentCommand{\K}{O{t}}{K^{#1}}
\DeclareDocumentCommand{\layer}{O{s}O{\N}}{\binom{#2}{#1}}
\DeclareDocumentCommand{\nlayer}{O{n} m}{\binom{[{#1}]}{#2}}
\DeclareDocumentCommand{\C}{O{s}}{\mathcal{C}_{#1}}

\DeclareDocumentCommand{\iseq}{m m}{n_{#1},n_{{#1}-1},\dots,n_{{#1}-{#2}+1}}
\DeclareDocumentCommand{\is}{O{s}O{s}O{\ell}}{[\iseq{#2}{#3}]_{#1}}
\DeclareDocumentCommand{\seq}{O{s}O{\ell}}{(\iseq{#1}{#2})}

\newcommand{\cK}{\mathcal{K}}
\newcommand{\cF}{\mathcal{F}}

\newcommand{\Abar}{\overline{\A}}
\newcommand{\cEbar}{\overline{\cE}}

\newcommand{\st}{\ \text{s.t.}\,}
\newcommand{\divides}{\mathrel{\bigm|}}

\declaretheorem[name=Theorem]{thm}
\declaretheorem[name=Lemma, sibling=thm]{lem}
\declaretheorem[name=Corollary, sibling=thm]{cor}

\declaretheorem[name=Definition, style=definition, sibling=thm]{defn}
\declaretheorem[style=definition]{problem}
\declaretheorem[name=Question, style=definition]{qu}
\declaretheorem[name=Remark, style=definition, sibling=thm]{rem}

\newcommand{\bigwo}{\mathbin{\big\backslash}}

\newcommand{\Bigwo}{\mathbin{\Big\backslash}}
\newcommand{\retlex}{<_{R}}
\newcommand{\A}{\mathcal{A}}
\newcommand{\B}{\mathcal{B}}
\newcommand{\cE}{\mathcal{E}}
\renewcommand{\H}{\mathcal{H}}

\newcommand{\J}{\mathcal{J}}

\renewcommand{\B}{\mathcal{B}}
\renewcommand{\R}{\mathcal{R}}

\newcommand{\U}{\mathcal{U}}

\newtoggle{kkt_details}
\togglefalse{kkt_details}

\begin{document}

\pagestyle{plain}

\title{Many cliques in bounded-degree hypergraphs}
\author{Rachel Kirsch\thanks{Department of Mathematical Sciences, George Mason University. Partially supported by NSF DMS-1839918.}
\and
Jamie Radcliffe\thanks{Department of Mathematics, University of Nebraska-Lincoln. Partially supported by Simons grant 429383.}}
\date{March 1, 2023}
\maketitle

\abstract{Recently Chase determined the maximum possible number of cliques of size $t$ in a graph on $n$ vertices with given maximum degree. Soon afterward, Chakraborti and Chen answered the version of this question in which we ask that the graph have $m$ edges and fixed maximum degree (without imposing any constraint on the number of vertices). In this paper we address these problems on hypergraphs. For $s$-graphs with $s\ge 3$ a number of issues arise that do not appear in the graph case. For instance, for general $s$-graphs we can assign degrees to any $i$-subset of the vertex set with $1\le i\le s-1$.

We establish bounds on the number of $t$-cliques in an $s$-graph $\H$ with $i$-degree bounded by $\D$ in three contexts: 
$\H$ has $n$ vertices; $\H$ has $m$ (hyper)edges; and (generalizing the previous case) $\H$ has a fixed number $p$ of $u$-cliques for some $u$ with $s\le u \le t$. When $\D$ is of a special form we characterize the extremal $s$-graphs and prove that the bounds are tight. These extremal examples are the shadows of either Steiner systems or partial Steiner systems. On the way to proving our uniqueness results, we extend results of F\"uredi and Griggs on uniqueness in Kruskal-Katona from the shadow case to the clique case.}

\section{Introduction}

There has been recent interest in generalized Tur\'an problems: determining the maximum (or minimum) number of copies of a fixed graph $T$ that a graph $G$ can contain, subject to a variety of constraints. The roots of this problem go back to Tur\'an's theorem \cite{Turan1941} and its extension by Zykov \cite{Z49} which determine, respectively, the maximum number of copies of $K_2$ and $K_t$ in a graph on $n$ vertices containing no $K_{r+1}$. The paper of Alon and Shikhelman \cite{AlonS} proved many foundational results and introduced the general problem to a wider audience.

\subsection{Many cliques in bounded-degree graphs}

We will focus on hypergraph versions of three generalized Tur\'an problems: determining the maximum number of cliques in graphs of bounded degree, using either vertices, edges, or cliques as a ``resource.'' We discuss the graph problems below; for a more complete history see \cite{CC20, chase20, CR14, CR2016, GLS, KR19, KR21}. The first phase of progress in these problems consisted of ``signpost'' results: estimates that are best possible infinitely often, but not for all values of the parameters.

We write $\k(G)$ for the number of cliques of size $t$ (and always insist that $t\ge 1$). Similarly $\k[\ge t](G)$ is the number of cliques of size at least $t$ in $G$. The next two theorems are versions of results due to Wood, phrased to match the hypergraph results we prove later.

\begin{thm}[Wood \cite{Wood2007}]\label{thm:folklore_vs}
	If $G$ is a graph on $n$ vertices with $\D(G)\le r-1$ then 
	\[
		\k(G)   \le \frac{n}r\, \binom{r}t          \quad\text{and}\quad
		\k[\ge 1](G) \le \frac{n}r\, \parens[\big]{2^r-1} 
	\]
	with equality when $G = aK_r$.
\end{thm}

\begin{thm}[Wood \cite{Wood2007}]\label{thm:folklore_edges}
	If $G$ is a graph having $m$ edges with $\D(G)\le r-1$ then 
	\[
		\k(G) \le \frac{m}{\binom{r}2} \, \binom{r}t       \quad\text{and}\quad \k[\ge2](G) \le \frac{m}{\binom{r}2} \, \parens[\big]{2^r-r-1}, 
	\]
	with equality when $G = aK_r$.
\end{thm}

Quite recently results in this direction were proved that are best possible for all values of the parameters. The vertex problem was solved by Chase \cite{chase20}. He proved a conjecture of Gan, Loh, and Sudakov \cite{GLS} using their reduction of the problem to the case $t=3$. Later Chao and Dong \cite{ChaoDong2022} gave a new proof of \cref{thm:chase} that proves the result for all $t$ simultaneously.

\begin{thm}[Chase \cite{chase20}, Chao and Dong \cite{ChaoDong2022}]\label{thm:chase}
	Let $G$ be a graph with $\D(G) \le r-1$ on $n$ vertices. Let $a$ and $b$ satisfy $n=ar+b$ with $0 \le b < r$. Then
    \[
        \k(G) \le a\binom{r}t + \binom{b}t,
    \]
    with equality for the graph $G = aK_{r} \cup K_b$, the disjoint union of $a$ copies of $K_r$ and one copy of $K_b$.
\end{thm}

Using \cref{thm:chase}, Chakraborti and Chen \cite{CC20} solved the edge problem. 

\begin{thm}[Chakraborti and Chen \cite{CC20}]\label{thm:cc}
    Let $G$ be a graph with $\D(G)\le r-1$ having $m$ edges. Let $a$ and $b$ satisfy $m = a\binom{r}{2}+ b$ with $0\le b < \binom{r}{2}$. Then
    \[
        \k(G) \le a\binom{r}t + \k(\C[2](b)),
    \]
    with equality for the graph $G = aK_r\cup \C[2](b)$. Here, $\C[2](b)$ is the colex graph having $b$ edges: the graph on vertex set $\N$ whose edges are the first $b$ pairs in colexicographic order.
\end{thm}

In this paper we are concerned with hypergraph versions of these problems. To state the questions we need to introduce our notation for hypergraphs and discuss the issue of degrees in hypergraphs. This we do next.  

In \cref{sec:KK} we discuss various versions of the Kruskal-Katona theorem, which is central in this area. In \cref{sec:signpost} we prove general results for arbitrary degree bounds. In \cref{sec:extremal} we introduce constructions which, in some cases, give optimal examples,  and prove some results about optimality and asymptotic optimality. Finally, in \cref{sec:open} we mention some open problems.

\subsection{Hypergraph definitions and questions}

Our notation is mostly standard.

\begin{defn}
	An \emph{$s$-graph} $\H$ is a pair $(V,\cE)$ consisting of a set of vertices $V$ together with a subset $\cE\of \layer[s][V]$. Frequently we'll suppress mention of the vertex set and simply use $\H$ to refer to the edge set. 
	If $I\of V$ has size $i$ then we define the \emph{neighborhood} $\H(I)$ of $I$ to be the $(s-i)$-graph with edge set 
\[
	\cE\parens[\big]{ \H(I)} = \setof{E\wo I}{I\of E \in \cE(\H)}.
\]
The \emph{degree of $I$ in $\H$} is the number of these edges, i.e.,
\[
	d_{\H}(I) = \abs[\big]{\setof{E \in \cE(\H)}{I\of E}}.
\]
We let the vertex set of $\H(I)$ be the union of all the edges in $\cE(\H(I))$, i.e., we omit all vertices not contained in an edge of $\H(I)$. The \emph{maximum $i$-degree} of $\H$ is simply
\[
	\D_i(\H) = \max\setof[\Big]{d_{\H}(I)}{I \in \layer[i][V]}.
\]
\end{defn}

We now define shadows and cliques in hypergraphs.

\begin{defn}
    Suppose that $\A$ is an $s$-graph. The \emph{shadow of $\A$ on level $q$} (where $q<s$) is given by
    \[
        \shad(\A) = \setof[\Big]{B}{\text{$\abs{B} = q $ and $\exists\, A\in \A \st B\of A$}} = \union_{A\in \A} \binom{A}{q}.
    \]
    The \emph{set of cliques on level $t$} (where $t>s$) is
    \[
        \K(\A) = \setof[\Big]{C}{\text{$\abs{C}=t$ and $\binom{C}{s} \of \A$}}.
    \]
    We let $\k(\A) = \abs[\big]{\K(\A)}$. 
\end{defn}

We can now state the questions we address in this paper.

\begin{qu}\label{qu:vertices}
    Suppose that an $s$-graph $\H$ has $n$ vertices, and that for some $1\le i \le s-1$ and $D > 0$ we have $\D_i(\H) \le D$. Given $t \ge s$, what is the maximum possible value of $\k(\H)$? In other words we aim to determine
    \[
        \max\setof{\k(\H)}{\text{$\H$ an $s$-graph with $n$ vertices and $\D_i(\H) \le D$}}.
    \]
\end{qu}

\begin{qu}\label{qu:edges}
    Suppose that an $s$-graph $\H$ has $m$ edges, and that for some $1\le i \le s-1$ and $D > 0$ we have $\D_i(\H) \le D$. Given $t \ge s$, what is the maximum possible value of $\k(\H)$? In other words, what is
    \[
        \max\setof{\k(\H)}{\text{$\H$ an $s$-graph with $m$ edges and $\D_i(\H) \le D$}}?
    \]
\end{qu}

\begin{qu}\label{qu:cliques}
	Suppose that an $s$-graph $\H$ has $\k[u](\H) = p$ for some $u \ge s$, and that for some $1\le i \le s-1$ and $D > 0$ we have $\D_i(\H) \le D$. Given $t \ge u$, what is the maximum possible value of $\k(\H)$? I.e., determine
	\[
	\max\setof{\k(\H)}{\text{$\H$ an $s$-graph with $\k[u](\H)=p$ and $\D_i(\H) \le D$}}.
	\]
\end{qu}

\subsection{Related extremal problems} 
\label{sub:prior_work}

The area of extremal problems for hypergraphs is rich and deep. The Kruskal-Katona theorem, which we discuss in \cref{sec:KK}, is an upper bound on the number of $t$-cliques in an $s$-graph with a given number of edges. Moreover, it implies a bound on the number of $t$-cliques in an $s$-graph having a given number of $u$-cliques for some $s < u\le t$. In \cite{Frohmader10}, Frohmader improved this bound in the case $s=2$.

The Kruskal-Katona theorem puts few restrictions on the $s$-graphs involved. A substantial amount of work has been done when we forbid large cliques in our $s$-graphs. The earliest such result is by Zykov \cite{Z49}. He proved the following result for graphs.

\begin{thm}[Zykov \cite{Z49}]\label{thm:Zykov}
    If $\H$ is a graph on $n$ vertices containing no $(r+1)$-clique then $\k(\H)\le \k(T_r(n))$. Here $T_r(n)$ is the \emph{Tur\'an graph}, that is to say it is the complete $r$-partite graph on $n$ vertices whose parts are of sizes as equal as possible. 
\end{thm}

The analogous result where we constrain $G$ to have $m$ edges is much more recent. The following result is due to Frohmader \cite{Frohmader}. To describe the result we need to define the $r$-partite colex Tur\'an graph. Let $r$ be a positive integer.  The \emph{$r$-partite colex order} is the restriction of the colex order on $\binom{\N}2$ to $\setof{ij}{i \not\equiv j \pmod{r}}$. The \emph{$r$-partite colex Tur\'an graph with $m$ edges}, $CT_r(m)$, is the graph on vertex set $\N$ whose edge set consists of the first $m$ edges in $r$-partite colex order. (Note that if $m = t_r(n)$, then the unique non-trivial component of $CT_r(m)$ is isomorphic to  $T_r(n)$.)  

\begin{thm}[Frohmader \cite{Frohmader}] \label{thm:EdgeZykov}
    If $G$ is a $K_{r+1}$-free graph with $m$ edges and $2 \leq t \leq r$, then $\k(G) \leq \k\bigl(CT_r(m)\bigr)$.
\end{thm}

In stark contrast to these positive results about graphs, even the Tur\'an problem for $s$-graphs with $s>2$ is apparently intractable. 
For no $r > s \ge 3$ is the problem of determining 
\[
	\max\setof{\abs{\H}}{\text{$\H$ is an $s$-graph on vertex set $[n]$ not containing an $(r+1)$-clique}}
\]
solved for all $n$, even asymptotically. (See Keevash's survey \cite{Keevash2011} for extensive discussion of this problem.) The hypergraph analogue of \cref{thm:EdgeZykov} seems no easier.

In a recent paper, Liu and Wang \cite{LiuWang2020} determined the maximum number of $t$-cliques in an $s$-graph on $n$ vertices containing at most $k$ disjoint edges (for $n$ sufficiently large). 

In the context of hypergraphs with bounded degree, Jung \cite{Jung21} considered the question of minimizing the ratio $\abs{\shad[s-1](\H)}/\abs{\H}$ for $s$-graphs $\H$ having bounded $1$-degree. Jung's results have a similar spirit to ours, but are not directly comparable. In an opposite direction F\"uredi and Zhao \cite{FZ2021} considered $3$-graphs $\H$ with large minimum degree and gave asymptotically best possible lower bounds on the size of $\shad[2](\H)$. 

\section{The Kruskal-Katona Theorem}
\label{sec:KK}

The fundamental theorem given in \cref{thm:kkt} below was proved independently by Kruskal \cite{K63} and Katona \cite{K68}. It shows that for a given number of edges $m$, the $s$-graph with the most $t$-cliques and the smallest $q$-shadow is the \emph{colex hypergraph}, denoted $\C(m)$, whose edges form an initial segment in the \emph{colexicographic} (or \emph{colex}) \emph{order}. Colex order is defined on finite subsets of $\N$ by $A < B$ iff $\max(A\symd B) \in B$. The original version of the Kruskal-Katona theorem discussed only shadows, but the version below describes also a closely related version, giving bounds on the number of cliques in $s$-graphs. For completeness we prove these versions (and slightly more) in \Cref{sec:kkd}. 

\begin{defn}
	We define the following functions mapping a number of edges $m$ to the  the size of the $q$-shadow  and the number of $t$-cliques of $\C(m)$.
	\[
		\shad^s(m) = \abs{\shad(\C(m))} \quad\text{and}\quad \k_s(m) = \k(\C(m)).
	\]
\end{defn}

\begin{restatable}[The Kruskal-Katona Theorem \cite{K68,K63}]{thm}{kruskalkatona}
    \label{thm:kkt}
	For all $0\le q < s < t \le n$, if $\A$ is an $s$-graph on vertex set $V$ with $\abs{V}=n$ then we have
	\iftoggle{kkt_details}{
        \[
		    \abs{\shad(\A)} \ge \shad^s(m), \qquad \k(\A) \le \k_s(m), \qquad\text{and}\qquad \abs{\upshad(\A)} \ge \abs{\upshad(\R_s(n,m))}
	    \]}
    {
        \[
		    \abs{\shad(\A)} \ge \shad^s(m), \qquad\text{and}\qquad \k(\A) \le \k_s(m), 
	    \]}
	where $m=\abs{\A}$. In other words, the colex $s$-graph $\C(m)$ has the smallest $q$-shadow and the largest number of $t$-cliques among all $s$-graphs of size $m$\iftoggle{kkt_details}{, whereas the smallest upshadow is achieved by the initial segment in the retlex order}{}. 
\end{restatable}

We also record here the following relationship between the functions $\k_s$ and $\shad[n-t]^{n-s}$.

\begin{restatable}{lem}{remthirtynine}\label{lem:ktcolex}
    For all $0\le s\le t \le n $ and $0\le m \le \binom{n}s$, 
    \[
        \k_s(m) = \binom{n}t - \shad[n-t]^{n-s} \parens[\big]{{\textstyle \binom{n}s} -m }.
    \]
\end{restatable}

\subsection{Cascade notation}
The standard way of describing initial segments of the colex order is \emph{cascade} notation, introduced by Kruskal in \cite{K63}. A good reference for the material in this subsection is Chapter 6 of the book \cite{FT18} by Frankl and Tokushige.

\begin{defn}
    We will say that an integer sequence $(n_s,n_{s-1},\dots,n_{s-\ell+1})$ is a $\emph{cascade}$ if it is strictly decreasing. We will define, for $s\ge 1$ and arbitrary cascades $(n_s,n_{s-1},\dots,n_{s-\ell+1})$ of length $\ell\ge 0$,
    \[
        \is = \sum_{k=0}^{\ell-1} \binom{n_{s-k}}{s-k}.
    \]
    We say that a cascade is a \emph{strict $s$-cascade} if $n_{s-k} \ge s-k$ for all $0\le k\le \ell-1$, and also $\ell\le s$. In that case every term in (the sum defining) $\is$ is positive. 
\end{defn}

\begin{rem} \label{rem:strict}
	In checking that a cascade $(n_s,n_{s-1},\dots,n_{s-\ell+1})$ is strict it is sufficient to check that $n_{s-k} \ge s-k$ for $k=\ell-1$, because if so then for every $k<\ell-1$ we have 
	\[
		n_{s-k} \ge n_{s-\ell+1} + (\ell-1-k) \ge s-\ell+1 + (\ell-1-k) = s-k.
	\]
\end{rem}
\begin{defn}
    If $\B$ is a family of sets, each disjoint from a fixed set $A$, we write $A+\B$ for the family
    \[
        A+\B = \setof{A\cup B}{B\in \B}.
    \]
\end{defn}

\begin{lem}\label{lem:uniquescascade}
    For all $m\ge 0$ and all $s\ge 1$ there exists a unique strict $s$-cascade such that $m=\is$. Indeed $\seq$ is the unique strictly decreasing sequence of length $\ell \ge 0$ satisfying
    \begin{align*}
            \binom{n_s}s                        &< m < \binom{n_s+1}{s} \\
            \binom{n_s}s + \binom{n_{s-1}}{s-1} &< m < \binom{n_s}s + \binom{n_{s-1}+1}{s-1} \\
                                                &\vdots & \\
            \binom{n_s}s + \binom{n_{s-1}}{s-1} + \cdots + \binom{n_{s-\ell+2}}{s-\ell+2} & < m <
            \binom{n_s}s + \binom{n_{s-1}}{s-1} + \cdots + \binom{n_{s-\ell+2}+1}{s-\ell+2}  \\
            \binom{n_s}s + \binom{n_{s-1}}{s-1} + \cdots + \binom{n_{s-\ell+1}}{s-\ell+1} & = m.
    \end{align*}
    
    If $\seq$ has length $1$ then the first of these inequalities is satisfied with equality on the left. If $m=0$ then we get the unique sequence of length $0$ for all $s\ge 1$. Moreover, for all $m\ge 0$ and $s\ge 1$ the colex initial segment of $\layer$ of length $m$ is
    \[
        \C(m) = \union_{k=0}^{\ell-1} \parens[\Bigg]{ \setof{n_{s-j}+1}{0\le j<k} + \nlayer[n_{s-k}]{s-k} }
    \]
    where $\seq$ is the unique $s$-cascade such that $m = \is$.
\end{lem}
	
\begin{defn}
	For all $m \ge 0$ and all $s \ge 1$, we denote by $i_s(m)$ the unique $s$-cascade such that $m = \is$, guaranteed by \cref{lem:uniquescascade}.

\end{defn}

Using cascade notation, we can exhibit lovely expressions for the number of cliques and the size of the shadow of a colex initial segment.
\begin{lem}\label{lem:facts}
	If $\seq$ is a strict $s$-cascade and $m=\is$ then
	\begin{align*}
		\k_s(m) = \k(\C(m))              &= \is[t], \text{ and} \\[1.5ex]
		\shad^s(m)= \abs{\shad(\C(m))}   &= \is[q].
	\end{align*} 
\end{lem}

\begin{proof}
Straightforward. See \cite{FT18} for a proof of the shadow case when $q=s-1$. The general shadow result and the proof for cliques are similar. Note that neither the $t$-cascade nor the $q$-cascade need be strict.
\end{proof}

\subsection{Lov\'asz Kruskal-Katona}
\label{sub:approximate_kruskal_katona}

Cascades have the merit of giving the precise values of $\shad^s(m)$ and $\k_s(m)$, but are somewhat unwieldy to work with. There is a simpler form of the Kruskal-Katona theorem, due to Lov\'asz \cite{LovBook},  that is often strong enough. We work with the natural polynomial generalization of the binomial coefficient $\binom{n}{k}$ to real values of $n$. 

\begin{defn}
	For a real number $x$ and natural number $k$, the generalized binomial coefficient is defined as $\binom{x}{k} = (x)(x-1)\cdots(x-k+1)/k!$. Note that $\binom{x}k$ is strictly increasing for $x\ge k-1$ and all $y\ge 0$ can be represented in the form $y=\binom{x}k$ for some $x\ge k-1$. 
\end{defn}

\begin{lem}[Lov\'asz \cite{LovBook}]\label{lem:LKK}
	Let $\H$ be an $r$-graph. If $\abs{\H}=\binom{u}{r}$, where $u \ge r$ is real, then $\abs{\shad[k](\H)} \ge \binom{u}{k}$ for all $k \in [r]$.
\end{lem}

The clique version of this result is a straightforward consequence.

\begin{thm}\label{thm:LKK}
	Let $s, t \in \N$ with $t \ge s$. Let $\H$ be an $s$-graph with $\abs{\H} = \binom{x}{s}$, where $x \ge s-1$ is real. Then if $x<t$ we have $\k(\H)=0$ and otherwise $\k(\H) \le \binom{x}{t}$.
\end{thm}
\begin{proof}
	If $x<t$ then $\abs{\H} < \binom{t}s$ and in particular $\H$ does not have enough edges to contain a $t$-clique, i.e., $\k(\H)=0$. If $x\ge t$, then let $\mathcal{T} = \K(\H)$, so $\mathcal{T}$ is a $t$-graph. We define $u\ge t$ by $\abs{\mathcal{T}}=\binom{u}{t}$. By \cref{lem:LKK}, the number of $s$-sets (edges of $\H$) contained in edges of $\mathcal{T}$ ($t$-cliques of $\H$) is at least $\binom{u}{s}$. The number of edges of $\H$ contained in $t$-cliques of $\H$ is at most the number of edges of $\H$, so we have $\binom{x}{s} = \abs{\H} \ge \binom{u}{s}$. Since $\binom{x}{s}$ is strictly increasing in $x$ for $x\ge s-1$ we must have $x \ge u\ge t$, so $\k(\H) = \abs{\mathcal{T}} = \binom{u}{t} \le \binom{x}{t}$.
\end{proof}

\section{Signpost Results for Hypergraphs}
\label{sec:signpost}

In this section we prove ``signpost'' versions of Theorems \ref{thm:chase} and \ref{thm:cc} for hypergraphs. We solve three related problems, fixing the numbers of vertices, edges, and cliques. For each problem we prove an upper bound on the number of $t$-cliques.

\subsection{Hypergraphs with a fixed number of vertices}

We start with a bound on the number of $t$-cliques in an $s$-graph on $n$ vertices with maximum degree at most $\D$. The argument bounds the number of cliques that can contain a fixed $i$-set $I$, and deduces a bound on the total number of $t$-cliques.

\begin{thm}\label{thm:Q2gen}
	Let $1 \le i < s$ and suppose that $\H$ is an $s$-graph on $n$ vertices such that $\D_i(\H) \le \D$. Then
	\[
		\k[t](\H) \le \binom{n}i \frac{\k[t-i]_{s-i}(\D)}{\binom{t}i} .
	\]
	If equality holds then for each $I\in \nlayer{i}$ the neighborhood $\H(I)$ contains $\k[t-i]_{s-i}(\D)$ $(t-i)$-cliques.
\end{thm}

\begin{proof}
	We count pairs $(I,K)$ where $I\in \nlayer{i}$, $K\in \K[t](\H)$, and $I\of K$. Counting by $t$-cliques in $\H$ we have a total of $\binom{t}i \, \k[t](\H)$. On the other hand consider $I\in \nlayer{i}$. For cliques $K$ that contain $I$ all $s$-sets $E$ such that $I\of E \of K$ must be in $\H$. Thus $\abs{\setof{K}{I\of K\in \K[t](\H)}} \le \k[t-i]\parens[\big]{\H(I)}$. Since by hypothesis $\abs{\H(I)}=d_\H(I) \le \D $ we have 
	\[
		\abs{\setof{K}{I\of K\in \K[t](\H)}} \le \k[t-i]\parens[\big]{\H(I)} \le \k[t-i]_{s-i}\parens{ \D}
	\]
	by \cref{thm:kkt}. Thus, summarizing, we have 
	\begin{align*}
		\binom{t}i \, \k[t](\H) &\le \binom{n}i \k[t-i]_{s-i}(\D) \label{} \\
					  \k[t](\H) &\le \binom{n}i \frac{\k[t-i]_{s-i}(\D) }{\binom{t}{i}} . 
	\end{align*}
	If we have equality then $\k[t-i]\parens[\big]{\H(I)} = \k[t-i]_{s-i}\parens{ \D}$ for every $I\in \nlayer{i}$.
\end{proof}

From this result the following corollary is immediate from our known bounds on $\k[t-i]_{s-i}$.

\begin{cor}\label{cor:explicit}
	Let $1 \le i < s $ and suppose that $\H$ is an $s$-graph on $n$ vertices such that $\D_i(\H) \le \D$.
	\begin{enumerate}[a),thmparts]
		\item If the cascade representation of $\D$ is given by $\is[s-i][s-i]$ then 
		\[
			\k[t](\H) \le \binom{n}i \frac{\is[t-i][s-i]}{\binom{t}i}.
		\]
		\item \label{part:approx} If $\D = \binom{x-i}{s-i}$ for some (not necessarily integral) $x\ge s$ then we have
		\[
			\k[t](\H) \le \binom{n}i \frac{\binom{x-i}{t-i}}{\binom{t}i} = \binom{n}i \frac{\binom{x}{t}}{\binom{x}i},
		\]
        if $x\ge t$ and $\k[t](\H)=0$ for $s\le x < t $.
	\end{enumerate}
\end{cor}

\begin{proof}
	The two parts follow from \cref{thm:Q2gen} together with \cref{lem:facts} and \cref{thm:LKK} respectively. 
\end{proof}

\subsection{Hypergraphs with a fixed number of edges}

We switch now to considering hypergraphs with a fixed number of edges.

We write $\K_{\H}(E)$ for the set of $t$-cliques in $\H$ containing the edge $E$ and $\k_{\H}(E)$ for $\abs{\K_{\H}(E)}$.

\begin{lem}\label{lem:scliquesum}
	For any $s$-graph $\H$ and $t \ge s$,
	\[
	\k(\H)\binom{t}{s} = \sum_{E \in \H}\k_\H(E).
	\]
\end{lem}

\begin{proof}
	Count the pairs $(E,K)$, where $E \of K \in \K(\H)$, in two ways.
\end{proof}

\begin{lem}\label{lem:nbhd}
Let $\H$ be an $s$-graph containing an edge $E \in \H$, and let $I \subsetneq E$ with $\abs{I} = i$. Let $K$ be a $t$-clique of $\H$ containing $E$. Then $K\wo I$ is a $(t-i)$-clique in $\H(I)$, and $\k_{\H}(E) \le k^{t-i}_{\H(I)} (E\wo I)$.
\end{lem}

\begin{proof}
    We'll show that $K \mapsto K\wo I$ is map from $\K_\H(E)$ to $\K[t-i]_{\H(I)}(E\wo I)$ from which it is clear that the map is an injection. We have $\abs{K\wo I} = t-i$ since $I \subsetneq E \of K$. Consider then an $(s-i)$-subset $F \of K\wo I$. We have $F \cup I \in \binom{K}{s} \of \H$, hence $F = (F\cup I)\wo I \in \H(I)$. Therefore $K\wo I \in \K[t-i](\H(I))$ and $K\wo I \in \K[t-i]_{\H(I)}(E\wo I)$.
\end{proof}

\begin{lem}\label{lem:nbhdcpereal}
	Let $1 \le i < s < t$ and suppose that $\H$ is an $s$-graph such that $\D_i(\H) \le \binom{x-i}{s-i}$ for some (not necessarily integral) $x \ge t-1$. 
	If $I \subsetneq E \in \H$ and $\J = \H(I)$, then 
	\[
	\frac{\k[t-i](\J)}{\abs{\J}} \le \frac{(x-s)_{(t-s)}}{(t-i)_{(t-s)}},
	\]
	where $\k[t-i](\J)$ is the number of $(t-i)$-cliques in the $(s-i)$-graph $\J$. If equality is achieved then $\abs{\J} = \binom{x-i}{s-i}$ and $k^{t-i}(\J) = \binom{x-i}{t-i}$. 
\end{lem}

\begin{proof}
	The number of edges in the neighborhood is $\abs{\J} = d_{\H}(I) \le \Delta_i(\H) \le \binom{x-i}{s-i}$, so $\abs{\J} = \binom{y}{s-i}$ for some $s-i-1 \le y \le x-i$. If $y < t-i$, then $k^{t-i}(\J) = 0$, so the lemma holds. Otherwise, $y \ge t-i$. By \cref{thm:LKK},  $k^{t-i}(\J) \le \binom{y}{t-i}$, so
	\begin{align*}
		\frac{k^{t-i}(\J)}{\abs{\J}} &\le \frac{\binom{y}{t-i}}{\binom{y}{s-i}}\label{tag1}\tag{1}\\
		&= \frac{y(y-1)\cdots(y-s+i+1)(y-s+i)\cdots(y-t+i+1)}{y(y-1)\cdots(y-s+i+1)}\frac{(s-i)!}{(t-i)!}\\
		&= \frac{(y-s+i)_{(t-s)}}{(t-i)_{(t-s)}}\quad\text{using }s < t\\
		&\le \frac{(x-s)_{(t-s)}}{(t-i)_{(t-s)}},\label{tag2}\tag{2}
	\end{align*}
	since $(x-s)_{(t-s)}$ is a strictly increasing function of $x$ for $x \ge t-1$, and we have $y+i > t-1$. If $\frac{\k[t-i](\J)}{\abs{\J}} = \frac{(x-s)_{(t-s)}}{(t-i)_{(t-s)}}$, then equality holds in (\ref{tag2}), so $y+i=x$, and $\abs{\J} = \binom{x-i}{s-i}$. Then equality in (\ref{tag1}) implies that $\k[t-i](\J) = \binom{x-i}{t-i}$. 
\end{proof}

\begin{rem}\label{rem:Jung} The expression $\k[t-i]_{s-i}(m)/m$ is not an increasing function of $m$, whereas $\frac{(x-s)_{(t-s)}}{(t-i)_{(t-s)}}$ is an increasing function of $x$. For values of $m$ where \[
	\frac{\k[t-i]_{s-i}(m)}{m} = \max_{m'\le m}\frac{\k[t-i]_{s-i}(m')}{m'}\label{tag3}\tag{3}
	\] we can improve \cref{lem:nbhdcpereal} to say that if $\D_i(\H) \le m$ then \[	\frac{\k[t-i](\J)}{\abs{\J}} \le \frac{\k[t-i]_{s-i}(m)}{m}. 
	\]	
	For $m = \binom{x-i}{s-i}$ where $x$ is an integer, it is easy to check that (\ref{tag3}) holds. It is an interesting question to determine which values of $m$ satisfy (\ref{tag3}).
\end{rem}

\begin{thm}\label{thm:edge1}
	Let $1 \le i < s$ and suppose that $\H$ is an $s$-graph having $m$ edges such that $\D_i(\H) \le \binom{x-i}{s-i}$ for some (not necessarily integral) $x \ge s$. Then, for all $t \ge s+1$, \[\k[t](\H) \le m\frac{\binom{x}t}{\binom{x}s}.
	\]
	If equality holds then for each $I \in \shad[i](\H)$ we have $\k[t-i](\H(I)) = \binom{x-i}{t-i}$.
\end{thm}

\begin{proof}
	If $t > x$ then $\k(\H) = 0$ because any $i$-set $I$ contained in a $t$-clique would have $d_{\H}(I) \ge \binom{t-i}{s-i} > \binom{x-i}{s-i}$. Therefore we may assume $t \le x$. We will count 
	\[
	S = \set{(I, E, K): I \subsetneq E \of K \in \K(\H), \abs{I} = i, \abs{E} = s}
	\] 
	in two ways. Counting by $K$, then $E$, then $I$, we obtain \[\abs{S} = \k(\H)\binom{t}{s}\binom{s}{i}.\] Counting by $I$, then $E$, then $K$, and letting $\J = \H(I)$, we obtain
	\begin{align*}
		\abs{S} &= \sum_{I \in \shad[i](\H)} \sum_{E \supseteq I} \k_\H(E)\\
		&\le \sum_{I \in \shad[i](\H)} \sum_{E \supseteq I} k^{t-i}_{\J} (E\wo I)\quad\text{by \cref{lem:nbhd}}\\
		&= \sum_{I \in \shad[i](\H)} k^{t-i}(\J)\binom{t-i}{s-i}\quad\text{by \cref{lem:scliquesum}}\\
		&\le \binom{t-i}{s-i} \sum_{I \in \shad[i](\H)} \frac{(x-s)_{(t-s)}}{(t-i)_{(t-s)}}\abs{\J}\quad\text{by \cref{lem:nbhdcpereal}}\\
		&= \binom{t-i}{s-i}\frac{(x-s)_{(t-s)}}{(t-i)_{(t-s)}} \sum_{I \in \shad[i](\H)} d_{\H}(I)\\
		&= \binom{t-i}{t-s}\frac{\binom{x-s}{t-s}}{\binom{t-i}{t-s}} \sum_{I \in \shad[i](\H)} d_{\H}(I)\\
		&=\binom{x-s}{t-s}\binom{s}{i}m.
	\end{align*}
	Therefore, $\k(\H)\binom{t}{s}\binom{s}{i} = \abs{S} \le \binom{x-s}{t-s}\binom{s}{i}m$,
	and
	\[
	\k(\H) \le \frac{\binom{x-s}{t-s}}{\binom{t}{s}}m = \frac{\binom{x}{t}}{\binom{x}{s}}m.
	\]
	The last equation follows from the fact that $\binom{x}{t}\binom{t}{s} = \frac{(x)(x-1)\cdots(x-t+1)}{s!(t-s)!} = \binom{x}{s}\binom{x-s}{t-s}$.
	
	If $\k[t](\H) = m\frac{\binom{x}t}{\binom{x}s}$ then we have equality in the above application of \cref{lem:nbhdcpereal} for every $I \in \shad[i](\H)$. By \cref{lem:nbhdcpereal}, $\k[t-i](\H(I)) = \binom{x-i}{t-i}$ for every $I \in \shad[i](\H)$.
\end{proof}

\subsection{Hypergraphs with a fixed number of cliques}

In this section we consider $s$-graphs that have a fixed number of $u$-cliques, for some $u > s$. The numbers of vertices and edges are not specified. We will use the following lemma to connect this problem to our previous results.

\begin{lem}\label{lem:degreebound}
	Let $1 \le i < s \le u$ and suppose that $\H$ is an $s$-graph such that $\Delta_i(\H) \le \binom{x-i}{s-i}$ for some (not necessarily integral) $x \ge s$. If $x<u$ then $\H$ has no $u$-cliques, and otherwise the $u$-graph $\U:=\K[u](\H)$ satisfies $\D_i(\U) \le \binom{x-i}{u-i}$.
\end{lem}

\begin{proof}
	For any $i$-set $I$ of vertices of $\H$, let $\cK = \U(I)$ and let $\cF = \H(I)$. We prove first that $\cK \subseteq \K[u-i](\cF)$. Consider an arbitrary $(u-i)$-edge  $E_I$ of $\cK$. By definition it satisfies $E_I\cup I \in \K[u](\H)$, so every $s$-set in $E_I\cup I$ is an edge of $\H$, and every $(s-i)$-set in $E_I$ is an edge of $\H(I)$. Therefore $E_I$ is a $(u-i)$-clique in $\H(I) = \cF$, as required.
    
    We are given that $\abs{\cF} = d_{\H}(I) \le \binom{x-i}{s-i}$, so we have $\abs{\cF} = \binom{y}{s-i}$ for some $s-i-1\le y \le x-i$. By \cref{thm:LKK},  if $y<u-i$ then $\k[u-i](\cF)=0$ , i.e., $\cK$ is empty, and otherwise $\k[u-i](\cF) \le \binom{y}{u-i} \le \binom{x-i}{u-i}$. If $x<u$ then we are always in the first case. Otherwise we have 
$d_\U(I) = \abs{\cK} \le \k[u-i](\cF) \le \binom{x-i}{u-i}$.
\end{proof}

We generalize \cref{thm:edge1} as follows. The $s=u$ case is exactly \cref{thm:edge1}.

\begin{thm}\label{thm:uedge1}
	Let $1 \le i < s \le u$ and suppose that $\H$ is an $s$-graph such that $\k[u](\H) = p$ and $\Delta_i(\H) \le \binom{x-i}{s-i}$ for some (not necessarily integral) $x \ge s$. Then, for all $t \ge u$, \[\k(\H) \le p\frac{\binom{x}{t}}{\binom{x}{u}}.\] If equality holds then for each $I \in \shad[i](\U)$ we have $\k[t-i](\U) = \binom{x-i}{t-i}$, where $\U = \K[u](\H)$.
\end{thm}

\begin{proof}
	By \cref{lem:degreebound}, we can apply \cref{thm:edge1} to the $u$-graph $\U:=\K[u](\H)$. Since $\U$ is a $u$-graph with $p$ edges and $\D_i(\U) \le \binom{x-i}{u-i}$, Theorem \ref{thm:edge1} implies that for all $t \ge u$ we have $\k[t](\U) \le p \frac{\binom{x}{t}}{\binom{x}{u}}$, with equality only if for each $I \in \shad[i](\U)$ we have $\k[t-i](\U(I)) = \binom{x-i}{t-i}$. Recall $s \le u \le t$. Given a $t$-clique $T$ in the $s$-graph $\H$, every $u$-set in $T$ is a $u$-clique of $\H$, so $T$ is also a $t$-clique in the $u$-graph $\U$. Therefore $\k[t](\H) \le \k[t](\U) \le p \frac{\binom{x}{t}}{\binom{x}{u}}$.
\end{proof}

\section{Extremal Hypergraphs and Asymptotic Tightness} 
\label{sec:extremal}

In this section we discuss the extent to which the signpost results from the previous section are tight. We begin in \Cref{ssec:ukkt} by discussing cases where colex hypergraphs are the unique examples achieving the bounds in \Cref{thm:kkt}. In \Cref{ssec:steiner} we then introduce some constructions that we use to produce cases of equality in our theorems. In the later subsections we discuss the three signpost results in relation to asymptotic tightness and uniqueness of examples. 

\subsection{Uniqueness in Kruskal-Katona}\label{ssec:ukkt}
We introduce two definitions from \cite{FurediGriggs86} by F\"uredi and Griggs.

\begin{defn}
    Given $1\le q < s \le n$ we say that $m$ is a \emph{jumping number} (or $(s,q)$-jumping number if we want to be more explicit) if $\shad^s(m+1) > \shad^s(m)$. We say that $m$ is a \emph{colex-unique number} if all $s$-graphs with $m$ edges satisfying $\abs{\shad(\H)} = \shad^s(m)$ are isomorphic to $\C(m)$.
\end{defn}

The following two theorems are proved in \cite{FurediGriggs86}.

\begin{thm}[F\"{u}redi and Griggs \cite{FurediGriggs86}] \label{thm:jumping}
	Suppose that $1 \le q < s \le n$ and that $0 \le m \le \binom{n}{s}$ is represented by the strict $s$-cascade $m = \is$. Then $m$ is an $(s,q)$-jumping number if and only if $\ell\le q$.
\end{thm}

\begin{thm}[F\"{u}redi and Griggs \cite{FurediGriggs86}] \label{thm:uniq}
	Suppose that $1 \le q < s \le n$ and that $0 \le m \le \binom{n}{s}$ is represented by the strict $s$-cascade $m = \is$. Then $m$ is a colex-unique number for all $m\le s+1$. If $m>s+1$ then $m$ is a colex-unique number if and only if one of the following is true:
	\begin{enumerate}[a),thmparts]
		\item \label{part:jumping} $m$ is a jumping number, i.e. $\ell\le q$, or
		\item \label{part:one_under} there exists $ n' \le n$ such that $m=\binom{n'}s-1$.
	\end{enumerate}
    For $m>s+1$ conditions \ref{part:jumping} and \ref{part:one_under} are mutually exclusive.
\end{thm}

The next lemma and the subsequent corollary will help us in the process of tracing the criterion for uniqueness through the steps of the proof of \cref{thm:kkt}.

\begin{lem} \label{lem:compl}
	Suppose that $u,v\ge 1$ and the cascade representations
	\[
		N = \is[u][u][k] \qquad\text{and}\qquad M = [m_v,m_{v-1},\dots,m_{v-\ell+1}]_v
	\]
	satisfy $n_{u-k+1} = m_{v-\ell+1}$. Let $b = n_{u-k+1} = m_{v-\ell+1}$. Suppose moreover that 
	\begin{align*}
		\set{b,n_{u-k+2},\dots,n_{u-1},n_u} \cup \set{b,m_{v-\ell+2},\dots,m_{v-1},m_v} 
		&= \set{b,b+1,\dots,u+v-1}, \\\shortintertext{and} 
		\set{b,n_{u-k+2},\dots,n_{u-1},n_u} \cap \set{b,m_{v-\ell+2},\dots,m_{v-1},m_v} 
		&= \set{b}.
	\end{align*}
	Then $N+M=\binom{u+v}{u}=\binom{u+v}{v}$.
\end{lem}

\begin{proof}
	Consider first the case that $\min(u,v)=1$. Without loss of generality we suppose that $u=1$. Then $u+v-1=v$ so for some $1\le b\le v$ we have
    \begin{align*}
        N+M &= [b]_1 + [v,v-1,v-2,\dots,b+1,b]_v \\
            &= b + \sum_{i=b}^v \binom{i}i \\
            &= b +(v-b+1) = v+1 = \binom{u+v}{u}.
    \end{align*}
	
	Now suppose that $u,v>1$. By symmetry we may suppose that $n_u=u+v-1$. If $k>1$ then we let
	\[
		N' = [n_{u-1},n_{u-2},\dots,b].
	\]
	Note that the representations of $N'$ and $M$ satisfy the hypotheses of the lemma, with $u'=u-1$ and $k'=k-1$. By induction we get 
	\[
		N+M = \binom{u+v-1}{u} + N'+M = \binom{u+v-1}{u} + \binom{u+v-1}{u-1} = \binom{u+v}{u}.
	\]
	On the other hand if $k=1$ then we're forced to have $N=[u+v-1]_u$ and $M=[u+v-1]_v$, so
	\[
		N+M = \binom{u+v-1}{u} + \binom{u+v-1}{v} = \binom{u+v-1}{u} + \binom{u+v-1}{u-1} = \binom{u+v}{u}. \qedhere
	\]
\end{proof}

\begin{cor} \label{cor:compl_len}
    Suppose that $1 \le s < n$ and that $0<m<\binom{n}{s}$. Let
    \[
        m = \is
    \]
    be the $s$-cascade representation of $m$. Then the $(n-s)$-cascade representation of $m' = \binom{n}{s} - m$ is 
    \[
        m' = [n'_{n-s},n'_{n-s-1},\dots,n'_{n-s-k+1}]_{n-s},
    \]
    where $n_{s-\ell+1} = n'_{n-s-k+1}$ and, writing $b$ for this value,
    \begin{equation*} \label{eq:compl}
        \begin{aligned}
        	\set{b, n_{s-\ell+2},\dots, n_{s-1},n_{s}} \cup \set{b,n'_{n-s-k+2},\dots,n'_{n-s-1},n'_{n-s}} &= \set{b,b+1,\dots,n-1} \\
        		\set{b, n_{s-\ell+2},\dots, n_{s-1},n_{s}} \cap \set{b,n'_{n-s-k+2},\dots,n'_{n-s-1},n'_{n-s}} &= \set{b}.
        \end{aligned}
        \tag{\dag}
    \end{equation*}
    In particular $k+\ell-1 = n-b$, so $k = n - \ell - b + 1$. 
\end{cor}

\begin{proof}
    With $n'_{n-s},n'_{n-s-1},\ldots,b$ defined to satisfy \cref{eq:compl} it is easy to check that $n'_{n-s-k+1} = b \ge n-s-k+1$ and  $k\le n-s$. Using \cref{rem:strict} we deduce that $(n'_{n-s},n'_{n-s-1},\cdots,b)$ is a strict $(n-s)$-cascade.
	Then, by \cref{lem:compl}, 
    \[
        \is + [n'_{n-s},n'_{n-s-1},\dots,n'_{n-s-k+1}]_{n-s} = \binom{s + (n-s)}{s} = \binom{n}s.
    \]
    Thus $[n'_{n-s},n'_{n-s-1},\dots,n'_{n-s-k+1}]_{n-s}$ is the $(n-s)$-cascade representation of $\binom{n}s-m$. 
\end{proof}

\begin{thm} \label{thm:clique_jumping}
    Suppose that $1 \le s < t \le n$ and that $0<m<\binom{n}{s}$. Let
    \[
        m+1 = \is
    \]
    be the $s$-cascade representation of $m+1$, having length $\ell$. Then $m$ has $\k_s(m+1) > \k_s(m)$ if and only if $t \le \ell + n_{s-\ell+1} -1$. In this case we say that $m$ is an $(s,t)$-clique-jumping number.
\end{thm}

\begin{proof}
    From \cref{lem:ktcolex} we have 
    \[
        \k(\C(m)) = \binom{n}t - \shad[n-t]^{n-s} \parens[\big]{{\textstyle \binom{n}s} -m }.
    \] 
    Thus $\k(m+1)>\k(m)$ exactly if we have 
    \[
        \shad[n-t]^{n-s}({\textstyle \binom{n}s-m)} > \shad[n-t]^{n-s}({\textstyle \binom{n}s}-m-1)
    \] 
    i.e., $\binom{n}s-m-1$ is an $(n-s,n-t)$-jumping number. By \cref{cor:compl_len}, the length of the $(n-s)$-cascade representation of $\binom{n}s-m-1$ is $k = n - \ell - n_{s-\ell+1} + 1$, so by \cref{thm:jumping} we need $n-\ell-n_{s-\ell+1}+1 \le n-t$, i.e., $t\le \ell + n_{s-\ell+1}-1$. 
\end{proof}

\begin{thm} \label{thm:clique_unique}
    Suppose that $1 \le s < t \le n$ and that $0<m<\binom{n}{s}$. Let
    \[ 
        m = \is
    \]
    be the $s$-cascade representation of $m$, having length $\ell$. Then the colex $s$-graph $\H = \C(m)$ is unique up to isomorphism satisfying $\abs{\H}=m$ and $\k(\H)=\k_s(m)$ if either $m\ge \binom{n}s - n+s-1$ holds, or $m< \binom{n}s - n+s-1$ and one of the following two (mutually exclusive) conditions holds: 
	\begin{enumerate}[a),thmparts]
        \item \label{part:compl_short} $t \le \ell+n_{s-\ell+1}-1$ (equivalently $m-1$ is an $(s,t)$-clique-jumping number), or
        \item for some $n-s+2\le n'\le n$ we have $\displaystyle m = \binom{n}s - \binom{n'}{n-s} +1$.
    \end{enumerate}
\end{thm}
\begin{proof}
	By \cref{lem:ktcolex}, the colex $s$-graph $\H = \C(m)$ is unique up to isomorphism satisfying $\abs{\H}=m$ and $\k(\H)=\k_s(m)$ if and only if all $(n-s)$-graphs with $\binom{n}{s}-m$ edges satisfying $\abs{\shad[n-t](\H)} = \shad[n-t]^{n-s}(\binom{n}{s}-m)$ are isomorphic to $\C[n-s](\binom{n}{s}-m)$. Applying \cref{thm:uniq}, and using \cref{cor:compl_len} and \cref{thm:clique_jumping} for condition a), yields the result. In condition b), note that $n' \le n - s + 1$ and $m = \binom{n}s - \binom{n'}{n-s} +1$ imply $m \ge \binom{n}s - n+s-1$.
\end{proof}

\begin{cor} \label{cor:super_short}
    If $m=\binom{n'}{s}$ with $n'\ge t$ then $\C(m)$ is the unique $s$-graph $\H$, up to isomorphism, with $m$ edges achieving $\k(\H) = \k_s(m)$. 
\end{cor}

\begin{proof}
    By \cref{thm:clique_unique}, it suffices to show that either  $\binom{n'}s \ge \binom{n}s-n+s-1$, or condition \ref{part:compl_short} is satisfied. For that condition note that $m = \is = [n']_s$ has length $\ell = 1$ and final entry $n_{s-\ell+1} = n'$, and we have $t \le 1 + n' - 1$ by hypothesis.
\end{proof}

\subsection{Steiner shadows and packing shadows} \label{ssec:steiner}

Here we define and discuss some important hypergraphs that turn out to be optimal examples in some cases of our problem.

\begin{defn}
    A \emph{Steiner system} with parameters $i,r,n$ (abbreviated as an $S(i,r,n)$) is a collection of $r$-sets of some $n$-set $V$ that covers each $i$-set of $V$ exactly once. 
    That is to say, it is an $r$-graph $\A$ on vertex set $V$ such that for all $I\in \binom{V}{i}$ there exists a unique $A\in \A$ such that $I\of A$. 
\end{defn}

    It has been known for a long time (by straightforward counting arguments) that in order for a Steiner system with parameters $i,r,n$ to exist it must be the case that certain divisibility conditions are satisfied. In groundbreaking work Peter Keevash \cite{Keevash} showed (among other things) that for sufficiently large $n$ these conditions are also sufficient.

\begin{thm}[Keevash \cite{Keevash}]\label{thm:keevash}
	For fixed $i\le r$ and for $n$ sufficiently large, an $S(i,r,n)$ exists if and only if for all $0\le j<i$ we have that $(r-j)_{(i-j)}$ divides $(n-j)_{(i-j)}$.
\end{thm}

\begin{cor}
		For fixed $i\le r$, the set of $n$ for which an $S(i,r,n)$ exists has positive lower density.
\end{cor}

\begin{proof}
    The divisibility conditions are certainly satisfied if $n-i+1$ is divisible by $r_{(i)}$, so the lower density of $\setof{n}{\text{an $S(i,r,n)$ exists}}$ is at least $1/r_{(i)}$.
\end{proof}

We can weaken the definition of a Steiner system to require  only that each $i$-set is covered at most once (rather than exactly once), giving the following definition. 

\begin{defn} An \emph{$i$-packing} of $r$-sets (abbreviated as a $P(i,r)$), also called a \emph{partial Steiner system}, is a collection of $r$-sets of some set $V$ that covers each $i$-set of $V$ at most once. That is to say, it is an $r$-graph $\A$ on vertex set $V$ such that for all $I\in \binom{V}{i}$ there exists at most one $A\in \A$ such that $I\of A$. Equivalently, any distinct $r$-sets $A, B \in \A$ have $\abs{A\cap B} < i$.
\end{defn}

Existence of $P(i,r)$'s is guaranteed for all values of the parameters. For instance, a disjoint collection of $r$-sets is a $P(i,r)$ for all $i \ge 1$.

The hypergraphs that will be useful to us are not only Steiner systems and packings themselves, but their shadows on layers intermediate between $i$ and $r$. 

\begin{defn}
	A \emph{Steiner shadow} with parameters $i,r,n,s$, abbreviated $\shad[s]S(i,r,n)$, is the $s$-shadow of an $S(i,r,n)$. A \emph{packing shadow} with parameters $i,r,s$, abbreviated $\shad[s]P(i,r)$, is the $s$-shadow of an $i$-packing of $r$-sets.
\end{defn}

We will show later that Steiner shadows and packing shadows provide examples showing that the signpost results we prove are best possible (at least for some values of the parameters). The following lemma computes relevant parameters of these hypergraphs.

\begin{lem}\label{lem:packing}
	If $1\le i < s < r$ and $\A$ is a $P(i,r)$, then, if we write $\H$ for the $s$-graph $\shad[s](\A)$, the following hold.
    \begin{enumerate}[a),thmparts]
        \item \label{part:layers} For all $i\le j \le r$ we have $\abs{\shad[j](\A)} = \binom{r}j \abs{\A}$. In particular, $\H$ has $\binom{r}s\abs{\A}$ edges, and for all $s \le t \le r$ we have $\k(\H) = \abs{\shad[t](\A)} = \binom{r}t \abs{\A}$.
        \item \label{part:cliques} If $I\in \shad[i](\H)$ then $\H(I)\cong K^{(s-i)}_{r-i}$, which implies that $d_\H(I) = \binom{r-i}{s-i}$ and $\k[t-i](\H(I)) = \binom{r-i}{t-i}$. In particular $\D_i(\H)=\binom{r-i}{s-i}$.
    \end{enumerate}
    In particular if $\H$ is a Steiner shadow $\shad[s]S(i,r,n)$ then parts \ref{part:layers} and \ref{part:cliques} hold with $\abs{\A} = \binom{n}i / \binom{r}{i}$, and $\shad[i](\H) = \nlayer{i}$.
\end{lem}

\begin{proof}
	Straightforward.
\end{proof}

We use the following lemma to prove the two corollaries following it: that two conditions on clique counts in neighborhoods force a hypergraph to be a packing shadow or a Steiner shadow respectively.

\begin{lem}\label{lem:extremal}
	Suppose that $i\ge 1$, that $i+2 \le s \le t \le r$, and that $\H$ is an $s$-graph with $\D_i(\H) \le \binom{r-i}{s-i}$. If $\H(I) \cong K^{(s-i)}_{r-i}$ for all $I \in \shad[i](\H)$, then $\H$ is a packing shadow $\shad[s]P(i,r)$.
\end{lem}

\begin{proof}
	For all sets $I \in \shad[i](\H)$ we write $A_I$ for the vertex set of $\H(I)$. Then $R_I=A_I\cup I$ has the property that for all $s$-sets $S \supseteq I$ we have $S\in \H$ if and only if $ S\of R_I$. We let $\R = \setof{R_I}{I\in \shad[i](\H)}$. We'll show that $\R$ is a $P(i,r)$ and that $\H=\shad[s](\R)$. 
	
	First let's show that if $I\in \shad[i](\H)$ and $J \in \binom{R_I}i$ then also $J\in \shad[i](\H)$ and $R_J=R_I$. We'll first prove the special case where $\abs{J\cap I}= i-1$. If $R_I \neq R_J$ then we can choose an $s$-set $S$ in $R_I$ containing $I\cup J$ and an element of $R_I\wo R_J$, since $s \ge i+2$. We have $I \of S \of R_I$, so $S \in \H$. Since $J\of S$ we have $J\in \shad[i](\H)$. Finally we have $J \of S \not\of R_J$, so $S \notin \H$. This contradiction implies that $R_I = R_J$. For any $J\in \binom{R_I}i$ there exists a sequence $I=J_0, J_1, \dots, J_k=J$ of $i$-sets of $R_I$ such that $\abs{J_\ell\cap J_{\ell+1}}=i-1$, and by the argument above we get that $R_{J_\ell}=R_I$ for all $\ell$. 
	
	From this we can show that if $I\in \shad[i](\H)$ then $\binom{R_I}{s} \of \H$. To see this, consider $S \in \binom{R_I}{s}$ and pick $J \in \binom{S}{i}$. Since $J\of S\of R_I=R_J$ we have $S\in \H$.
	
	Finally, set $\R=\set{R_I: I \in \shad[i](\H)}$ as above. To show that $\R$ is a $P(i,r)$, suppose $R_I$ and $R_{I'}$ are both in $\R$, and $J \of R_I\cap R_{I'}$ is an $i$-set. Then by the result in the second paragraph $R_I=R_J=R_{I'}$.  
	The last thing we need to show is that $\H=\shad[s](\R)$. If $S\in \H$ then for any $i$-set of $S$ we have $I\of S\of R_I$, so $S\in \shad[s](\R)$. On the other hand if $S\in \shad[s](\R)$ then there exists $I\in \shad[i](\H)$ with $S\of R_I$ and hence $S\in \H$ by the result in the third paragraph. 
\end{proof}

\begin{cor}\label{cor:extremalpacking}
    Suppose that $i\ge 1$, that $i+2 \le s \le t \le r$, and that $\H$ is an $s$-graph with $\D_i(\H) \le \binom{r-i}{s-i}$. If we have $\k[t-i](\H(I)) = \binom{r-i}{t-i}$ for every $i$-set $I$ contained in an edge of $\H$, then $\H$ is a packing shadow $\shad[s]P(i,r)$.
\end{cor}

\begin{proof} \cref{cor:super_short} implies that for all $I \in \shad[i](\H)$ we have $\H(I) \cong K^{(s-i)}_{r-i}$. \cref{lem:extremal} completes the proof.
\end{proof}

The corresponding result for Steiner shadows also follows.
\begin{cor} \label{cor:extremalsteiner}
    Suppose that $i\ge 1$, that $i+2 \le s \le t \le r$, and that $\H$ is an $s$-graph with $\D_i(\H) \le \binom{r-i}{s-i}$. If we have $\k[t-i](\H(I)) = \binom{r-i}{t-i}$ for every $i$-set $I$ of vertices of $\H$, then $\H$ is a Steiner shadow $\shad[s]S(i,r,n)$.
\end{cor}

\begin{proof}
	Let $V$ be the vertex set of $\H$. Given $I \in \binom{V}i$ we have $\k[t-i](\H(I)) = \binom{r-i}{t-i}$ and $t \le r$, so $\k[t-i](\H(I)) \ge 1$. Thus $\shad[i](\H)=\binom{V}i$. By \cref{cor:extremalpacking}, $\H$ is a packing shadow $\shad[s]P(i,r)$ with $\shad[i](\H)=\binom{V}i$, i.e a Steiner shadow $\shad[s]S(i,r,n)$, where $n=\abs{V}$.
\end{proof}

\subsection{Equality cases for the three problems}

Here we characterize the extremal hypergraphs for some cases of each of the three problems from \Cref{sec:signpost}. All the cases we discuss are ones where $\H$ is an $s$-graph and $\D_i(\H) \le \binom{r-i}{s-i}$ for some $r\ge s$ and $i<s$. We also show, for all three problems, that for these particular degree bounds our results are asymptotically tight.

\subsubsection{Hypergraphs with a fixed number of vertices} 
\label{ssub:vertices}

For degree bounds of the form $\binom{r-i}{s-i}$, with $r$ an integer, we show that Steiner shadows achieve the bound from \cref{thm:Q2gen}, and that they are the only $s$-graphs that do when $i \le s-2$. We do not know whether other $s$-graphs achieve the bound when $i = s-1$.

\begin{thm}\label{thm:achieves}
	Let $1 \le i < s \le t \le r$, where $r$ is an integer, and suppose that $\H$ is an $s$-graph on $n$ vertices.
	\begin{enumerate}[a)]
		\item If $\H$ is a Steiner shadow $\shad[s]S(i,r,n)$, then $\D_i(\H) = \binom{r-i}{s-i}$ and $\k(\H) = \frac{\binom{n}i}{\binom{r}i} \binom{r}{t}$. I.e., $\H$ achieves the upper bound in \cref{thm:Q2gen}.
		\item If we further assume that $s \ne i+1$, then $\H$ satisfies both $\D_i(\H) \le \binom{r-i}{s-i}$ and $\k(\H) = \frac{\binom{n}i}{\binom{r}i} \binom{r}{t}$ if and only if $\H$ is a Steiner shadow $\shad[s]S(i,r,n)$.
	\end{enumerate}
\end{thm}
Note that by \cref{thm:keevash} the set of $n$ for which Steiner shadows $\shad[s]S(i,r,n)$ exist has positive lower density.

\begin{proof}
	For both parts, note that $\k[t-i]_{s-i}(\binom{r-i}{s-i}) = \binom{r-i}{t-i}$ and $\binom{r-i}{t-i}\binom{r}{i} = \binom{r}{t}\binom{t}{i}$ (as in \cref{cor:explicit}), so $\frac{\binom{n}i}{\binom{r}i} \binom{r}{t} =\binom{n}i \frac{\k[t-i]_{s-i}(\D)}{\binom{t}i}$.
	
	First, suppose $\H = \shad[s](\A)$, where $\A$ is an $S(i,r,n)$. By \cref{lem:packing}, 
	$\H$ has $\D_i(\H)=\binom{r-i}{s-i}$ and $\k(\H) =  \frac{\binom{n}i}{\binom{r}i} \binom{r}{t}$.
	
	Now, suppose $s\ne i+1$ (so $3 \le i+2 \le s$) and $\H$ is an $s$-graph on $n$ vertices such that $\D_i(\H) \le \binom{r-i}{s-i}$ and $\k(\H) = \frac{\binom{n}i}{\binom{r}i} \binom{r}{t}$. 
	By the condition for equality in \cref{thm:Q2gen}, for each $I\in \nlayer{i}$ the neighborhood $\H(I)$ contains $\binom{r-i}{t-i}$ $(t-i)$-cliques, and so by \cref{cor:extremalsteiner}, $\H$ is a Steiner shadow $\shad[s]S(i,r,n)$.
\end{proof}

Now we show that the upper bounds given by \cref{thm:Q2gen} and \cref{cor:explicit} are asymptotically tight. We make use of the famous result where R\"odl's nibble was first introduced. 

\begin{thm}[R\"odl \cite{rodl}]\label{thm:rodl}
	The maximum number of edges in an $i$-packing of $r$-sets in $[n]$ is  $(1-o_n(1))\frac{\binom{n}{i}}{\binom{r}{i}}$.
\end{thm}

\begin{thm}\label{thm:o1vx}
	For $1 \le i < s \le t \le r \le n$, let $N$ be the maximum value of $\k(\H)$ over all $s$-graphs $\H$ on $n$ vertices with $\D_i(\H) \le \binom{r-i}{s-i}$. Then
	\[
	N = (1-o_n(1))\frac{\binom{n}i}{\binom{r}i} \binom{r}{t} .
	\]
\end{thm}

\begin{proof}    
    Let $\A$ be an $i$-packing of $r$-sets in $V$ with $\abs{\A} = (1-o_n(1))\frac{\binom{n}{i}}{\binom{r}{i}}$, as guaranteed by \cref{thm:rodl}. Then $\H = \shad[s](\A)$ has $\k(\H) = \abs{\A}\binom{r}{t} = (1-o_n(1))\frac{\binom{n}{i}}{\binom{r}{i}}\binom{r}{t}$ by \cref{lem:packing}. For every $I \in \binom{V}{i}$, we have \[d_{\H}(I) = \begin{cases}\binom{r-i}{s-i} & \text{if }I \in \shad[i](\A)\\0&\text{otherwise,}\end{cases}\]
	so $\D_i(\H) \le \binom{r-i}{s-i}$. Together with \cref{thm:Q2gen} this implies that
	$N = (1-o_n(1))\frac{\binom{n}{i}}{\binom{r}{i}}\binom{r}{t}$.
\end{proof}

In the proof of \cref{thm:o1vx}, $\A$ covers $(1-o_n(1))\binom{n}{i}$ of the $i$-sets in $V$, i.e. almost all of them, so there exists $\H$ that is almost a Steiner shadow and almost attains the upper bound. In particular it seems highly plausible that a stability version of \cref{thm:achieves} holds.

\begin{rem}\cref{thm:keevash} gives an alternative proof of \cref{thm:o1vx}.
\end{rem}

\subsubsection{Hypergraphs with a fixed number of edges}

For degree bounds of the form $\binom{r-i}{s-i}$, with $r$ an integer, we show that packing shadows achieve the upper bound in \cref{thm:edge1}, and that for $i \le s-2$, they are the only $s$-graphs that achieve this bound. Again, we do not know whether only packing shadows achieve the bound when $i = s-1$. 

\begin{thm}\label{thm:edge2}
	Let $1 \le i < s \le t \le r$, where $r$ is an integer, and suppose that $\H$ is an $s$-graph having $m$ edges.
	\begin{enumerate}[a)]
		\item If $\H$ is a packing shadow $\shad[s]P(i,r)$, then $\D_i(\H) = \binom{r-i}{s-i}$ and $\k[t](\H) = m\frac{\binom{r}t}{\binom{r}s}$. I.e.,  $\H$ achieves the upper bound in \cref{thm:edge1}. In particular, if $\binom{r}{s} \divides m$, then $\H = \frac{m}{\binom{r}{s}}K^{(s)}_r$ achieves equality.
		\item If we further assume that $s \ne i+1$, then $\H$ satisfies both $\D_i(\H) \le \binom{r-i}{s-i}$ and $\k[t](\H) = m\frac{\binom{r}t}{\binom{r}s}$ if and only if $\H$ is a packing shadow $\shad[s]P(i,r)$.
	\end{enumerate}
\end{thm}

\begin{proof}
	First, suppose $\H = \shad[s](\A)$, where $\A$ is a $P(i,r)$. By Lemma \ref{lem:packing}, $\D_i(\H) = \binom{r-i}{s-i}$, and we have $m = \abs{\A}\binom{r}{s}$ and $\k(\H) = \abs{\A}\binom{r}{t}$, so $\k(H) = m\frac{\binom{r}{t}}{\binom{r}{s}}$.
	
	If $\binom{r}{s} \divides m$, then $\frac{m}{\binom{r}{s}}K^{(r)}_r$ is a $P(i,r)$, and its $s$-shadow is $\frac{m}{\binom{r}{s}}K^{(s)}_r$. Note that \[\k\Big(\frac{m}{\binom{r}{s}}K^{(s)}_r\Big) = \frac{m}{\binom{r}{s}}\k\big(K^{(s)}_r\big) =\frac{m}{\binom{r}{s}}\binom{r}{t}\] and $\D_i\big(\frac{m}{\binom{r}{s}}K^{(s)}_r\big) = \binom{r-i}{s-i}$.
	
	Now, suppose $s \ne i+1$ (so $3 \le i+2 \le s$) and $\H$ is an $s$-graph having $m$ edges with $\D_i(\H) \le \binom{r-i}{s-i}$ and $\k[t](\H) = m\frac{\binom{r}{t}}{\binom{r}{s}}$. We have equality in the statement of \cref{thm:edge1}. The last sentence of \cref{thm:edge1} shows that $\k[t-i](\H(I)) = \binom{r-i}{t-i}$ for every $I \in \shad[i](\H)$. By \cref{cor:extremalpacking}, $\H$ is a packing shadow $\shad[s]P(i,r)$.
\end{proof}

The bound given by Theorem \ref{thm:edge1} is asymptotically tight.

\begin{thm}\label{thm:o1edge}
	For $1\le i < s \le t \le r$ and $m \ge 1$, let $M$ be the maximum value of $\k(\H)$ over all $s$-graphs $\H$ having $m$ edges with $\D_i(\H) \le \binom{r-i}{s-i}$. Then \[M = (1-o_m(1))m\frac{\binom{r}t}{\binom{r}s}.\]
\end{thm}

\begin{proof}
	Given $i,s,t,r,m$, let $m = a\binom{r}{s}+b$, for $0 \le b < \binom{r}{s}$. Then $M \ge \k(aK^{(s)}_t) = a\binom{r}{t} =  (1-\frac{b}{m})m\frac{\binom{r}{t}}{\binom{r}{s}}$. Since $0 \le b < \binom{r}{s}$, $\lim_{m\to\infty}\frac{b}{m} = 0$, so $M \ge (1-o_m(1))m\frac{\binom{r}{t}}{\binom{r}{s}}$. Theorem \ref{thm:edge1} implies $M \le m\frac{\binom{r}{t}}{\binom{r}{s}}$, completing the proof.
\end{proof}

\subsubsection{Hypergraphs with a fixed number of cliques}

When the degree bound is of the form $\binom{r-i}{s-i}$, with $r$ an integer, we show that the upper bound given by \cref{thm:uedge1} is achieved by any $s$-graph $\H$ for which the edges that contribute to the $u$-clique count of $\H$ form a packing shadow. By excluding the case $s=u$, which is addressed in \cref{thm:edge2}, we find that these are the only $s$-graphs that achieve this bound. The case $s=i+1$ is included here. In particular, when $s=i+1$, all degree bounds $\D \ge t-i$ are of the form $\binom{r-i}{s-i}$ for some $r \ge t$, so are covered by \cref{thm:uedge2}.

\begin{thm}\label{thm:uedge2}
Let $1 \le i < s < u \le t \le r$, where $r$ is an integer, and suppose that $\H$ is an $s$-graph with $\D_i(\H) \le \binom{r-i}{s-i}$. Let $p = \k[u](\H)$. Then $\k(\H) = p\frac{\binom{r}{t}}{\binom{r}{u}}$ if and only if the set of edges of $\H$ that are contained in a $u$-clique of $\H$ is a packing shadow $\shad[s]P(i,r)$. In particular, if $\binom{r}{u} \divides p$, then $\H = \frac{p}{\binom{r}{u}}K^{(s)}_r$ achieves equality.
\end{thm}

\begin{proof} 
	First, let $\cE = \setof{E \in \H}{E \subset U \text{ for some } U \in \K[u](\H)}$, and suppose $\cE = \shad[s](\A)$, where $\A$ is a $P(i,r)$. Note $\k[u](\cE) = \k[u](\H)$. Any edges in $\H \setminus \cE$ are not contained in $u$-cliques of $\H$ so cannot be contained in $t$-cliques of $\H$. Therefore $\k(\cE) = \k(\H)$. By Lemma \ref{lem:packing}, we have $p = \k[u](\H) = \abs{\A}\binom{r}{u}$, and $\k[t](\H) = \abs{\A}\binom{r}{t}$, so $\k[t](\H) = p\frac{\binom{r}{t}}{\binom{r}{u}}$.
	
		If $\binom{r}{u} \divides p$, then $\frac{p}{\binom{r}{u}}K^{(r)}_r$ is a $P(i,r)$, and its $s$-shadow is $\frac{p}{\binom{r}{u}}K^{(s)}_r$. Note that \[\k\Big(\frac{p}{\binom{r}{u}}K^{(s)}_r\Big) = \frac{p}{\binom{r}{u}}\k\big(K^{(s)}_r\big) =\frac{p}{\binom{r}{u}}\binom{r}{t}\] and $\D_i\big(\frac{p}{\binom{r}{u}}K^{(s)}_r\big) = \binom{r-i}{s-i}$.
	
	Now, suppose $\k(\H) = p\frac{\binom{r}{t}}{\binom{r}{u}}$. By \cref{lem:degreebound}, the $u$-graph $\U:=\K[u](\H)$ satisfies $\D_i(\U) \le \binom{r-i}{u-i}$. The last sentence of \cref{thm:uedge1} states that for each $I \in \shad[i](\U)$ we have $\k[t-i](\U) = \binom{r-i}{t-i}$. By \cref{cor:extremalpacking}, $\U$ is a packing shadow $\shad[u]P(i,r)$. Let $\A$ be a $P(i,r)$ such that $\U = \shad[u](\A)$. Since $\K[u](\H) = \U$, every edge $S$ of $\H$ that is contained in a $u$-clique $U$ of $\H$ is in $\shad[s](\A)$, because there is some $r$-set $R \in \A$ such that $S \of U \of R$.
\end{proof}

Theorem \ref{thm:uedge1} is asymptotically tight, by a proof very similar to that of Theorem \ref{thm:o1edge}.

\begin{thm}\label{thm:o1clique}
	For $1\le i < s \le u \le t \le r$ and $p \ge 1$, let $P$ be the maximum value of $\k(\H)$ over all $s$-graphs $\H$ having $\k[u](H) = p$ with $\D_i(\H) \le \binom{r-i}{s-i}$. Then \[P = (1-o_p(1))p\frac{\binom{r}t}{\binom{r}u}.\]
\end{thm}

\subsubsection{A theorem on $2$-graphs}

We also obtain the following corollary giving the maximum number of $t$-cliques among $2$-graphs with a fixed number of $u$-cliques and an arbitrary constant upper bound on the maximum degree.

\begin{thm}\label{cor:graph}
		Suppose $3 \le u \le t \le r$ and $G$ is a graph such that $\k[u](G) = p$ and $\D(G) \le r-1$. Then
		\begin{enumerate}[a)]
			\item $\k(G) \le p\binom{r}{t}/\binom{r}{u}$.
			\item The maximum value of $\k(G)$ over all such graphs is $(1-o_p(1))p\binom{r}t/\binom{r}u$.
			\item We have $\k(G) = p\binom{r}{t}/\binom{r}{u}$ if and only if $G$ (after removing any edge not contained in a $u$-clique) is a $(p/\binom{r}{u})K_r$ (possibly together with some isolated vertices). In particular, we have equality if and only if $\binom{r}{u} \divides p$.
		\end{enumerate}   
\end{thm}

\begin{proof}Apply \cref{thm:uedge1}, \cref{thm:uedge2}, and \cref{thm:o1clique} with $s=2$ and $i=1$. Note that a packing $P(1,r)$ is a set of disjoint $r$-sets, so its $2$-shadow forms a set of disjoint $r$-cliques.
\end{proof}

\Cref{cor:graph} is a signpost answer to a question in the concluding remarks of \cite{CC20}.

\section{Open Problems}\label{sec:open}

Many interesting problems still remain. We list some of them here.

\begin{problem}If $\D_i(\H) \le \binom{r-i}{s-i}$, where $r$ is an integer, \cref{thm:achieves}, \cref{thm:edge2}, and \cref{thm:uedge2} completely characterize the $s$-graphs that achieve the upper bounds given by \cref{thm:Q2gen} and \cref{thm:edge1} for $i \le s-2$, and \cref{thm:uedge1} for $u \ne s$. In particular, these upper bounds cannot be achieved for some values of the problem parameters.
	 \begin{enumerate}[a)]
	 	\item For values of $i$, $r$, and $n$ for which Steiner systems $S(i,r,n)$ do not exist (either because they do not satisfy the necessary divisibility conditions or because $n$ is too small---see \cref{thm:keevash}), 
	 	\cref{thm:achieves} shows that all $s$-graphs $\H$ on $n$ vertices having $\D_i(\H) \le \binom{r-i}{s-i}$ have $\k(\H) < \binom{n}{i}\binom{r}{t}/\binom{r}{i}$, although by \cref{thm:o1vx}, $\max \set{\k(\H)} = (1-o_n(1)) \binom{n}{i}\binom{r}{t}/\binom{r}{i}$. Which such $s$-graphs have the maximum number of $t$-cliques?
	 	\item By \cref{lem:packing}, if $\H = \shad[s](\A)$, with $\A$ a $P(i,r)$, then $\abs{\H} = \k[s](\H) = \binom{r}{s}\abs{\A}$. Therefore, by \cref{thm:edge2}, when $m \nmid \binom{r}{s}$, all $s$-graphs $\H$ having $m$ edges and $\D_i(\H) \le \binom{r-i}{s-i}$ have $\k(\H) < m\binom{r}{t}/\binom{r}{s}$, although by \cref{thm:o1edge}, $\max \set{\k(\H)} = (1-o_m(1))m\binom{r}{t}/\binom{r}{s}$. Which such $s$-graphs have the maximum number of $t$-cliques?
	 	\item Similarly, by \cref{thm:uedge2}, when $p \nmid \binom{r}{u}$, all $s$-graphs having $\k[u](\H) = p$ and $\D_i(\H) \le \binom{r-i}{s-i}$ have $\k(\H) < p\binom{r}{t}/\binom{r}{u}$, although by \cref{thm:o1clique}, $\max \set{\k(\H)} = (1-o_p(1))p\binom{r}{t}/\binom{r}{u}$. Which such $s$-graphs have the maximum number of $t$-cliques?
	 \end{enumerate}
 \end{problem}

\begin{problem} Among $s$-graphs with $\D_{s-1}(\H) \le r-s+1$ (the $s=i+1$ case) we have determined the exact maximum number of $t$-cliques and found extremal $s$-graphs.
	\begin{enumerate}[a)]
		\item Are there $s$-graphs $\H$ on $n$ vertices with $\D_{s-1}(\H) \le r-s+1$ that have $\k(\H) = \frac{\binom{n}{s-1}}{\binom{r}{s-1}} \binom{r}{t}$ but are not Steiner shadows $\shad[s]S(s-1,r,n)$?
		\item Are there $s$-graphs $\H$ on $m$ edges with $\D_{s-1}(\H) \le r-s+1$ that have $\k(\H) = m\frac{\binom{r}t}{\binom{r}s}$ but are not packing shadows $\shad[s]P(s-1,r)$?
	\end{enumerate}
\end{problem}

\begin{problem}  We have characterized the extremal $s$-graphs and proved that our upper bounds are asymptotically tight only when the $i$-degree bound is $\binom{r-i}{s-i}$ for some integer $r$. Are the upper bounds given by \cref{cor:explicit}, \cref{thm:edge1}, and \cref{thm:uedge1} tight when the $i$-degree bound does not have this form?
\end{problem} 
	
\begin{problem} For which values of $m$ does $\frac{\k[t]_{s}(m)}{m} = \max_{m'\le m}\frac{\k[t]_{s}(m')}{m'}$? (See \cref{rem:Jung}.)	\end{problem}

\begin{appendices}
    \section{Kruskal-Katona Details}
    \label{sec:kkd}

    In this appendix we give proof details for some of the results in  \Cref{sec:KK}. We prove a slightly expanded version of \Cref{thm:kkt}; one that discusses upshadows as well as (down) shadows and cliques. To state this result we define another total order on finite subsets.

    \begin{defn}\label{def:retlex}
        The \emph{retrolexicographic} (or \emph{retlex}) order on finite subsets of $\N$ is defined by $A \retlex B$ iff $\max(A\symd B)\in A$. We write $\R_s(n,m)$ for the $\retlex$-initial segment of size $m$ in $\nlayer{s}$.
        
        In addition, given a ground set $[n]$ and an $s$-graph $\A$ on $[n]$, we define 
        \[
            \Abar = \setof{[n]\wo A}{A \in \A},
        \]
        an $(n-s)$-graph on $[n]$ with the same size as $\A$.
    \end{defn}

    \begin{rem}\label{rem:retlex}  
        The definition has the following symmetries with the colex order.
        \begin{enumerate}[a),thmparts]
            \item\label{rem:retlex1} We have $A \retlex B$ if and only if $A > B$, i.e., retlex is the reverse of colex order. 
            \item\label{rem:retlex2} If both $A$ and $B$ are subsets of $[n]$ then, since $A\symd B = ([n]\wo A) \symd ([n]\wo B)$, we have $[n]\wo A \retlex [n]\wo B$ if and only if $A < B$. 
        \end{enumerate}
    In particular for $0 \le m \le \binom{n}{s}$ we have 
    \[
        \R_s\parens[\Big]{n,\binom{n}s-m} = \nlayer{s} \Bigwo \C(m) \qquad\text{and}\qquad \overline{\R_s(n,m)} = \C[n-s](m). 
    \]
    \end{rem}
     
    Note that colex initial segments are independent of $n$ (provided $m\le \binom{n}s$), whereas retlex initial segments depend in an essential way on $n$. Since there are many nice presentations of the bound on shadows (see for instance \cite{FT18}) we will only prove the clique and upshadow bounds.

    \toggletrue{kkt_details}

    \kruskalkatona*

    \begin{proof}
        We may assume without loss of generality that $V=[n]$. We'll start by proving the upshadow bound from the shadow bound. Given $\cE\of \nlayer{s}$ and writing $\cEbar = \setof{[n]\wo E}{E\in \cE} \of \nlayer{n-s}$, we have
    \begin{align*}
            \shad[n-t](\cEbar) &= \setof[\big]{[n]\wo T}{\abs{T} = t \text{ and } \exists\;([n]\wo E) \in \cEbar \st ([n]\wo T)\of([n]\wo E)}\\
            &= \setof[\Big]{[n]\wo T}{T\in \nlayer{t} \text{ and } \exists\, E \in \cE \st E \of T} =  \overline{\upshad(\cE)}.
    \end{align*}
        Thus, by the shadow bound, to minimize $\abs{\upshad(\cE)} = \abs{\overline{\upshad(\cE)}}$ we can take $\cEbar$ to be a colex initial segment, i.e., by \cref{rem:retlex} \ref{rem:retlex1}, $\cE$ to be a retlex initial segment. Now, for the clique bound, note that 
        \[
            \K(\A) = \nlayer{t} \Bigwo \upshad\parens[\Big]{\nlayer{s} \bigwo \A}.
        \]
        Thus to maximize $\abs{\K(\A)}$ we can take $\nlayer{s} \wo \A$ to be a retlex initial segment, i.e., by \cref{rem:retlex} \ref{rem:retlex2}, take $\A$ to be a colex initial segment.  
    \end{proof}

    Using \cref{rem:retlex} we can immediately read out of the proof of the previous theorem the functions $\k_s$ and $\shad[n-t]^{n-s}$.

    \remthirtynine*

    \begin{proof}
        We have 
        \begin{align*}
            \K(\C(m)) &= \nlayer{t} \Bigwo \upshad\parens[\Big] {\nlayer{s} \bigwo \C(m)} \\
                      &= \nlayer{t} \Bigwo \upshad\parens[\Big] {\R_s\parens[\big]{n,{\textstyle \binom{n}s}-m}} \\
                      &= \nlayer{t} \Bigwo \overline{\shad[n-t]\parens[\Big] { \overline{\R_s\parens[\big]{n,{\textstyle \binom{n}s}-m} }}} \\
                      &= \nlayer{t} \Bigwo \overline{\shad[n-t]\parens[\Big] { \C[n-s]\parens[\big]{{\textstyle \binom{n}s} - m }}} \\
            \shortintertext{i.e.,}
            \k_s(m) = \k(\C(m)) &= \binom{n}t - \shad[n-t]^{n-s} \parens[\big]{{\textstyle \binom{n}s} -m }. \qedhere
    \end{align*}
    \end{proof}
\end{appendices}

\bibliographystyle{amsplain}
\bibliography{Bibliography}
\end{document}

%% file: basics.tex
\newcommand{\of}{\subseteq}
\newcommand{\wo}{\setminus} 

\newcommand{\union}{\bigcup}
\newcommand{\symd}{\bigtriangleup}

\newcommand{\0}{\emptyset}

\newcommand{\D}{\Delta}

\newcommand{\N}{\mathbb{N}}

\newcommand{\R}{\mathbb{R}}

\newcommand{\C}{\mathbb{C}}

\DeclarePairedDelimiter\abs{|}{|}
\DeclarePairedDelimiter\parens{(}{)}
\DeclarePairedDelimiter\set{\{}{\}}

\DeclarePairedDelimiterX\setof[2]{\{}{\}}{#1\,:\,#2}